\documentclass[11pt]{article}
\usepackage{amsmath,amssymb,amsthm,amsfonts,verbatim}
\usepackage{enumerate}
\usepackage{microtype}
\usepackage{tikz-cd}
\usepackage{mathtools}
\usepackage{todonotes}
\usepackage{color}
\usepackage{framed}
\usepackage[top=1.2in,bottom=1.2in,left=1.21in,right=1.21in]{geometry}

\title{Finite covers of graphs, their primitive homology,\\ and representation theory}

\author{Benson Farb and Sebastian Hensel \thanks{The first author
    gratefully acknowledges support from the National Science
    Foundation.} }
    
\theoremstyle{plain}
\newtheorem{theorem}{Theorem}[section]
\newtheorem*{theorem*}{Theorem}
\newtheorem{question}[theorem]{Question}

\newtheorem{proposition}[theorem]{Proposition}
\newtheorem*{proposition*}{Proposition}
\newtheorem{lemma}[theorem]{Lemma}
\newtheorem*{lemma*}{Lemma}

\newtheorem{observation}[theorem]{Observation}
\newtheorem{remark}[theorem]{Remark}

\newtheorem{corollary}[theorem]{Corollary}
\theoremstyle{definition}
\newtheorem{definition}[theorem]{Definition}

\newcommand{\nc}{\newcommand}
\nc{\dmo}{\DeclareMathOperator}

\nc{\Q}{\mathbb{Q}}
\nc{\R}{\mathbb{R}}
\nc{\Z}{\mathbb{Z}}
\nc{\C}{\mathbb{C}}
\nc{\cS}{\mathcal{S}}
\nc{\iso}{\cong}
\dmo{\Mod}{Mod}
\dmo{\Diff}{Diff}
\dmo{\Homeo}{Homeo}
\dmo{\dist}{dist}
\dmo\BDiff{BDiff}
\dmo\SO{SO}
\dmo\slide{sl}
\dmo\im{im}
\dmo\Irr{Irr}
\dmo\Irrpr{Irr^{pr}}
\dmo\Irrscc{Irr^{scc}}
\dmo\Fix{Fix}
\dmo\Out{Out}
\dmo\Aut{Aut}
\dmo\rank{rank}
\dmo\Res{Res}
\renewcommand{\epsilon}{\varepsilon}
\nc{\coloneq}{\mathrel{\mathop:}\mkern-1.2mu=}
\nc{\margin}[1]{\marginpar{\scriptsize #1}}
\nc{\para}[1]{\bigskip\noindent\textbf{#1}}

\begin{document}
\maketitle
\begin{abstract}
Consider a finite, regular cover $Y\to X$ of finite graphs, with associated deck group $G$.  
We relate the topology of the cover to the structure of $H_1(Y;\C)$ as a $G$-representation.  A central object in this study is the {\em primitive homology} group $H_1^{\mathrm{prim}}(Y;\C)\subseteq 
H_1(Y;\C)$, which is the span of homology classes represented by components of lifts of primitive elements of $\pi_1(X)$.  This circle of ideas relates combinatorial group theory, surface topology, and representation theory.
\end{abstract}
{\small\tableofcontents}

\section{Introduction}
Consider a regular covering $f:Y\to X$ of finite graphs, with
associated deck group $G$.  The goal of this article is to better understand $H_1(Y)$ as a $G$-representation.  

The action of $G$ on $Y$ by homeomorphisms endows 
$H_1(Y;\C)$ with the structure of a $G$-representation.  
Gasch\"{u}tz (see \cite{GLLM}) observed that the Chevalley-Weil formula \cite{CW} for surfaces has the following analogue in this context:  there is an isomorphism of $G$-representations

\begin{equation}
\label{eq:cw:graphs}
H_1(Y;\C) \iso \C[G]^{n-1}\oplus \C
\end{equation}
where $\C[G]$ denotes the regular representation and $n=\rank(\pi_1(X))$ is the rank of the free group $\pi_1(X)$.

The isomorphism \eqref{eq:cw:graphs} hints at a broader dictionary
between representation-theoretic information (on the right-hand side)
and topological information (on the left-hand side).  One of the goals of this paper is to begin an exploration of this dictionary. 

\para{Primitive homology of $G$-covers. }
Of course $H_1(Y;\C)$ is spanned by a finite set of closed loops.  But what if we demand that these loops project under $f$ to be {\em primitive} in $\pi_1(X)$; that is, part of a free basis of the free group $\pi_1(X)$?  Fixing a cover $f:Y\to X$ we define, following 
Boggi-Looijenga \cite{BL} in the surface case (see \S~\ref{sec:surfaces}), the 
{\em primitive homology} of $Y$ to be:
\[H_1^{\mathrm{prim}}(Y;\C):=\text{$\C$-span}\{[\gamma]\in H_1(Y;\C): \text{$f(\gamma)\in \pi_1(X)$ is primitive}\}\]

By construction, $H_1^{\mathrm{prim}}(Y;\C)$ is a
$G$-subrepresentation of $H_1(Y;\C)$.  The main question we consider in this paper is the following, which we view as a fundamental question about finite covers of finite graphs. 

\begin{question}[{\bf Determining primitive homology}]
\label{question:subrep1}
What is the $G$-subrepresentation $H_1^{\mathrm{prim}}(Y;\C)\subseteq H_1(Y;\C)$?  In particular, is $H_1^{\mathrm{prim}}(Y;\C)=H_1(Y;\C)$ for every normal cover $f:Y\to X$?
\end{question}

In order to state Question~\ref{question:subrep1} more precisely, we let \[\Irr(G):=\{V_0:=\C_{triv},
V_1,\ldots ,V_r\}\] be the set of (isomorphism classes of) complex, irreducible
$G$-representations. In this notation, the Chevalley-Weil formula \eqref{eq:cw:graphs} can be
restated as:

\[H_1(Y;\C)\cong \bigoplus_{V_i\in\Irr(G)}V_i^{(n-1)\dim(V_i)}\oplus\C_{triv}.\]
Since $H_1^{\mathrm{prim}}(Y;\C)$ is a $G$-subrepresentation of
$H^1(Y;\C)$, Question~\ref{question:subrep1} is equivalent to the following.

\begin{question}[{\bf Which irreps occur?}]
\label{question:which1}
Which irreducible $G$-representations $V_i$ occur in
$H_1^{\mathrm{prim}}(Y;\C)$? What are their multiplicities? 
\end{question}

In order to describe our progress on answering
Question~\ref{question:which1} we make the following definition.  Note
that the data of a normal $G$-cover $f:Y\to X$ is equivalent to a
surjective homomorphism $\phi:\pi_1(X)=F_n\to G$.  The following can
therefore be viewed as a purely group-theoretic definition.

\begin{definition}[{\boldmath$\Irrpr(\phi, G)$}]
  Let $G$ be a finite group and let $\phi: F_n\to G$ be a homomorphism.
  Let $\Irrpr(\phi, G)\subset\Irr(G)$ denote the set of irreducible representations $V$ of
  $G$ with the property that $\phi(\gamma)(v)=v$ for some primitive
  element $\gamma\in F_n$ and some nonzero $v\in V$.
\end{definition}


Our first main result gives a restriction on $H_1^{\mathrm{prim}}(Y;\C)$ in terms of the purely algebraic data of $\Irrpr(\rho,G)$.

\begin{theorem}[{\bf Restricting primitive homology}]
\label{theorem:restriction}
Let $f:Y\to X$ be a normal $G$-covering of finite graphs defined by $\phi:F_n\to G$, where $G$ is
a finite group and $\rank(\pi_1(X))\geq 2$.  Then
\[ H_1^{\mathrm{prim}}(Y;\C) \subseteq \bigoplus_{V_i\in \Irrpr(\phi,G)}
V_i^{(n-1)\cdot\dim(V_i)}\oplus \C_{triv}\]
\end{theorem}

The proof of Theorem~\ref{theorem:restriction} is given in
\S\ref{sec:primitive} below and it uses surface topology.  After
finishing this work, we learned that Malestein and Putman
independently discovered this obstruction as well.

\smallskip
One can also obtain finer information about the structure of
$H_1^{\mathrm{prim}}(Y;\C)$ by studying the primitive homology of
intermediate (nonregular) covers; see \S\ref{sec:primitive} for
details. 
Theorem~\ref{theorem:restriction} suggests a strategy to search for
examples with $H_1^{\mathrm{prim}}(Y;\C) \neq H^1(Y;\C)$ via an
algebraic question about $G$-representations.

\begin{question}
\label{question:equals}
Given a group $G$, does $\Irrpr(\phi, G)=\Irr(G)$ for every surjective homomorphism $\phi:F_n\twoheadrightarrow G$?
\end{question}
We attack Question~\ref{question:equals} from various angles, by
restricting the class of groups we consider.  In \S\ref{sec:abelian} we answer Question~\ref{question:equals} in the
affirmative for $G$ abelian or $2$--step nilpotent. One might thus 
expect that $H_1^{\mathrm{prim}}(Y;\C) = H^1(Y;\C)$ in these cases.  We can prove this in 
many instances.

\begin{theorem}[{\bf Abelian and \boldmath$2$-step nilpotent covers}]
\label{theorem:nilpotent}
Let $Y\to X$ be a finite normal cover with deck group $G$ defined by $\phi:F_n\to G$.  
Assume that $\rank(\pi_1(X))\geq 3$.   Suppose that $G$ is either abelian or $2$-step nilpotent.  Then $\Irrpr(\phi, G)=\Irr(G)$.

In the case of a nonabelian $G$, assume further that $n$ is odd and every subgroup of the center of $G$ has prime order. Then 
\[ H_1^{\mathrm{prim}}(Y;\C)=H^1(Y;\C).\] 
\end{theorem}

The assumption $\rank(\pi_1(X))\geq 3$ is in general necessary, even for reasonably simple $G$.  
In  \S\ref{sec:rank-2-examples} we give a number of different examples that imply the following.

\begin{theorem}[{\bf Rank two counterexamples}]
\label{theorem:counter}
Let $X$ be a wedge of two circles. There exist finite $2$--step nilpotent groups $G$ and $G$-covers $Y\to X$ given by surjections $\phi:\pi_1(X)=F_2\to G$ so that $\Irrpr(\phi, G)\subsetneq\Irr(G)$ and $H_1^{\mathrm{prim}}(Y;\C) \subsetneq H^1(Y;\C)$.
\end{theorem}

\para{Primitives in the kernel.} If $\phi:F_n\to G$ has any primitive
element in the kernel, then trivially $\Irrpr(\phi, G) = \Irr(G)$ (and
in fact, primitive homology is all of homology for the associated
cover).  Grunewald and Lubotzky \cite{GL} study covers induced by such
maps and call them \emph{redundant}.

To find interesting examples in light of
Question~\ref{question:equals}, we are thus lead to a search for
$\phi$ without primitive elements in the kernel.

We discuss the ``no primitives in the kernel'' condition in \S\ref{sec:kernelprims}, and relate it in 
\S\ref{sec:product-replacement} to the Product Replacement Algorithm.  We 
connect this property to those discussed above, as follows : for a finite normal cover $Y\to X$ given by a surjection $\phi:\pi_1(X)=F_n\to G$ :

\bigskip
\noindent
\fbox{$\ker(\phi) \ \text{contains a primitive in $F_n$}$} $\Longrightarrow$\  \fbox{$H_1^{\mathrm{prim}}(Y;\C) = H^1(Y;\C)$}\  $\Longrightarrow$\fbox{$\Irrpr(\phi, G)=\Irr(G)$} 

\bigskip
The first implication is Lemma~\ref{lem:regular-intermediate} below; the second is Theorem~\ref{theorem:restriction}.  

In general, the property that $\ker(\phi)$ does not contain any
primitive element is far from sufficient to imply
$\Irrpr(\phi, G)\neq \Irr(G)$; for example, the standard mod-$2$ homology cover has
the former property, but not the latter.
In contrast, the following is easy to show from the definitions.
\begin{observation}\label{obs:freely-acting}
  Suppose that $G$ is a finite group which acts freely and linearly on
  a sphere, and let $\phi:F_n\to G$ be any homomorphism. Then
  $\Irrpr(\phi, G)=\Irr(G)$ if and only if $\ker(\phi)$ contains a
  primitive element.
\end{observation}
At first glance, this makes the class of groups which act freely and
linearly on spheres a likely candidate to answer
Question~\ref{question:equals} in the negative.
However, in Section~\ref{sec:free-actions} we prove the following, which 
shows that in general this strategy does not work.

\begin{theorem}[{\bf Groups acting freely on spheres}]
\label{theorem:spheres}
  There is a number $B\geq3$ with the following property. Suppose that $G$ is a finite group 
  that acts freely and linearly on an odd-dimensional sphere.  Then for all $n\geq B$ and all surjective homomorphisms 
  $\phi:F_n \to G$, the kernel of $\phi$ contains a primitive element of $F_n$.  Thus 
  $\Irr(G)=\Irrpr(\phi,G)$ for all $\phi$ and $H_1^{\mathrm{prim}}(Y;\C)=H^1(Y;\C)$ for every such $G$-cover.
\end{theorem}

We remark that the finite groups $G$ that have a free linear action on some
sphere have been classified; see e.g. \cite{W}, \cite{DM} or
\cite{N}. We use this classification in our proof of Theorem~\ref{theorem:spheres}.

\para{Algorithms.} In \S\ref{sec:algorithms} we consider the search for a surjection $\phi:F_n\to G$ with $\Irr(G) \neq \Irrpr(\phi,G)$ from an algorithmic perspective.  
A priori, the question if $\Irr(G) = \Irrpr(\phi,G)$ requires
knowledge about an infinite set (of all primitives) in the free
group. However, we will prove (Proposition~\ref{prop:alg}) one can algorithmically compute the set of all
elements of $G$ that are images of primitive elements under $\phi$. This allows computer-assisted searches for examples 
with $\Irr(G) \neq \Irrpr(\phi,G)$.

Concerning Theorem~\ref{theorem:spheres} above, experiments suggest that
the constant $B$ is likely very small, maybe
even $3$. We have checked that $\Irrpr(\phi, G)=\Irr(G)$ for all such $G$ and $\phi$ with 
$|G|\leq 1000$.   On the other hand, in Proposition~\ref{prop:examples-primitive} we use a computer-assisted search to give an example that shows $B$ must be at least $3$.  

\paragraph{The case of surfaces.} There is an analogous theory to the above with the finite graph $X$ replaced by a surface $S_g$, 
and primitive elements in $\pi_1(X)$ replaced by simple closed curves in $\pi_1(S_g)$.   The ``simple closed curve homology'' is closely related to the problem of vanishing (or not) of the virtual first Betti number of the moduli space of genus $g\geq $ Riemann surfaces.  We sketch this theory in 
\S\ref{section:surfaces}.

\paragraph*{Acknowledgements.} The authors would like to thank Andy
Putman and Justin Malestein for interest in the project and comments.

\section{Representation Theory of $H_1^{\mathrm{prim}}$}
\label{sec:primitive}

In this section we prove Theorem~\ref{theorem:restriction}, and give a
partial converse. We begin by setting up some notation that we will
use throughout the article. $X$ will denote the wedge of $n$ copies of
$S^1$.

Given a regular covering $Y\to X$, we fix a \emph{preferred} lift $\hat{x}_0$
of the basepoint $x_0\in X$. Given any loop $\gamma$ in $X$, the \emph{preferred elevation} $\hat{\gamma}$
of $\gamma$ is the lift at $\hat{x}_0$ of the smallest power $\gamma^k$ which does lift
(with degree $1$) to $Y$. Any image of $\hat{\gamma}$ under an element of the
deck group will be called an \emph{elevation} of $\gamma$.

We denote by $G_\gamma$ the stabilizer in $G$ of the preferred elevation of $\gamma$.  Note that $G_\gamma$ is cyclic.  It is generated by the image of $[\gamma]\in\pi_1(X)$ under the surjection $\pi_1(X)\to G$.  We denote this element by $g_\gamma$.

The key step in the proof of Theorem~\ref{theorem:restriction} is
the following proposition, which can be considered as a new entry in the
dictionary discussed in the introduction. 
\begin{proposition}\label{prop:rep-lift-of-prim}
  Let $X$ be a wedge of $n$ copies of $S^1$.  Let $Y\to X$ be a regular cover with 
  deck group $G$. Let $l$ be a primitive loop on $X$ and let $\widetilde{l}$ be
  the preferred elevation of $l$ to $Y$. Then there is an isomorphism of $G$--representations:  
  \[ 
  \mathrm{Span}_{H_1(X)}\left\{ g\cdot [\widetilde{l}]: g \in G] \right\}
  \cong
  \mathrm{Ind}_{G_\ell}^G\C_{\mathrm{triv}}
  \]
\end{proposition}

The proof of Proposition~\ref{prop:rep-lift-of-prim}  uses surface topology. 
 
\begin{proof}[Proof of Proposition~\ref{prop:rep-lift-of-prim}]
We prove the proposition in three steps.

\medskip
\noindent
{\bf Step 1 (Reduction to the surface case): }  
We choose an identification $X$ with the core graph of a
$(n+1)$--bordered sphere $S$ so that $l$ is freely homotopic to a
simple closed curve $\alpha$ on $S$. This is possible since
$\mathrm{Out}(F_n)$ acts transitively on the set of (conjugacy classes
of) primitive elements of $F_n$. In addition, we can assume that 
each component of $S-\alpha$ contains at least one boundary
component of $S$.

Let $F \to S$ denote the cover defined by
$\pi_1(Y) < \pi_1(X) = \pi_1(S)$. Note that $F$ is $G$--equivariantly
homotopy equivalent to $Y$. Let $\hat{\alpha}_i$ be the elevations of
$\alpha$ in $F$. The homology classes $[\hat{\alpha}_i]$ are exactly
the homology classes defined by the elevations of $l$ in $Y$. 

To prove Proposition~\ref{prop:rep-lift-of-prim}, it therefore suffices
to consider the case of a cover $F\to S$ of surfaces with boundary,
deck group $G$, and and $l$ a simple closed curve $\alpha$ with elevations $[\hat{\alpha}_i]$ with the property that each component of $S-\alpha$ contains at least one boundary component of 
$S$. We assume this setup for the rest of the proof of the proposition.

\medskip
\noindent
{\bf Step 2 (Independence of Elevations): }  Let $F$ be a surface with boundary, and let $\hat{\alpha}_1, \ldots,  \hat{\alpha}_n$ be disjoint, pairwise nonisotopic simple closed
  curves on $F$. Suppose that each complementary component of
  $\hat{\alpha}_1\cup \ldots \cup \hat{\alpha}_n$ in $F$ contains at
  least one boundary component of $F$.  We claim that $\{[\hat{\alpha}_1], \ldots, [\hat{\alpha}_n]\}\subset  H_1(F;\C)$ is linearly independent.

To see this, let $R_1,\ldots, R_k$ be the complementary components of
  $\hat{\alpha}_1\cup \ldots\cup \hat{\alpha}_n$ in $F$.  Let $\delta_i\subset
  R_i$ be the multicurve consisting of the union of all boundary components of $F$
  contained in $R_i$. Suppose that 
\begin{equation}
\label{eq:primlift1}
 a_1[\hat{\alpha}_1] + \ldots + a_n[\hat{\alpha}_n] = 0.   
 \end{equation}
 
 We can assume that all $a_i \neq 0$.
  Suppose that $\hat{\alpha}_1\cup \ldots, \cup \hat{\alpha}_n$
  separates the surface, since otherwise there is nothing to show. Thus, 
  without loss of generality, we can assume that $\hat{\alpha}_n$ lies in $\partial R_1\cap\partial R_2$.  By adding a suitable multiple of
  $\partial R_2$ if necessary, we can rewrite \eqref{eq:primlift1} as 
  \[ a'_1[\hat{\alpha}_1] + \ldots + a'_{n-1}[\hat{\alpha}_{n-1}] +
  b_2[\delta_2] = 0 \]
  for some $a'_i$ and some $b_2$.   Since $\hat{\alpha}_n$ does not in this equation, the
  support of this new relation contains $R_1 \cup R_2$ in its
  complement. Denote by $R_1'$ the complementary component containing
  $R_1$. We can now repeat the argument: unless $\hat{\alpha}_1,
  \ldots, \hat{\alpha}_{n-1}$ has become nonseparating, there will be
  some curve which lies on the boundary of $R_1'$ and a second
  boundary component $R_j$ for some $j$. Repeating this modification a
  finite number of times we end up with 
  \[ A_1[\hat{\alpha}_1] + \ldots + A_l[\hat{\alpha}_l] +
  b_2[\delta_2] + \ldots + b_k[\delta_k]= 0 \]
  where now $\hat{\alpha}_1, \ldots, \hat{\alpha}_l$ is
  nonseparating. Note that $[\delta_1]$ does not appear in this
  expression, and thus all involved classes are linearly
  independent. Thus, all coefficients are $0$, which implies that the
  original linear combination was trivial as well. This proves the claim.

\medskip 
\noindent
{\bf Step 3 (Identification of the representation): }By construction, the element $g_\alpha \in G$ fixes the preferred
elevation $\hat{\alpha}_1$ and permutes the other elevations
$\hat{\alpha}_i$ in the same exact way that it permutes the the cosets of $G_\alpha=\langle g_\alpha \rangle$ in  $G$.  By the standard characterization of induced representations and the
linear independence of the classes $[\hat{\alpha}_i]$, 
Proposition~\ref{prop:rep-lift-of-prim} follows.  
\end{proof}

\begin{proof}[Proof of Theorem~\ref{theorem:restriction}]
For any $G$-representations $U,W$, denote by $\langle U,W\rangle$ the inner product of the characters of the representations $U$ and $W$.  Let $V$ be any irreducible $G$-representation. Proposition~\ref{prop:rep-lift-of-prim} followed by Frobenius Reciprocity gives:

  \[
  \begin{array}{ll}
  \langle  \mathrm{Span}_{H_1(Y)}\left\{[\hat{\alpha}_1], \ldots, [\hat{\alpha}_n] \right\},V\rangle _G& =\langle \mathrm{Ind}_{G_\alpha}^G\C_{\mathrm{triv}}, V \rangle _G\\
  &\\
    & = \langle
  \mathrm{Res}_{G_\alpha}^GV,
  \C_{\mathrm{triv}}\rangle_{G_\alpha}\\
  &\\
  & =\dim(\Fix(G_\alpha)). 
  \end{array}\] 
  
  Thus an irreducible representation $V$ appears in
  $\mathrm{Span}_{H_1(Y)}\left\{ [\hat{\alpha}_1], \ldots,
    [\hat{\alpha}_n] \right\}$ if and only if $G_\alpha$ has a nonzero
  fixed vector in $V$ (equivalently, since $G_\alpha$ is cyclic, generated by $g_\alpha$, this is equivalent to $g_\alpha$ having a non-fixed vector).  
  Since every $G$-representation is a direct sum of irreducible representations, this
  implies that an irreducible representation $V$ appears in
  $H_1^{\mathrm{prim}}(Y)$ only if $V\in\Irrpr(\phi,G)$.
\end{proof}

The rest of this section is devoted to a criterion which can in theory be used
to determine the multiplicity of a given $V_i \in \Irrpr$ in the 
subrepresentation $H^{\mathrm{prim}}_1$.

To begin, we first note the following simple consequence of transfer.  
  \begin{lemma}
    Let $Y \to X$ be a regular $G$--covering, where $G$ is a finite group and $X$ is a graph
    with $\rank(\pi_1(X))\geq 2$. Let $g\in G$ be any element. Then
    \[ H_1(Y;\C)^{<g>} \iso H_1(Y/\langle g\rangle;\C) \]
    The same is true for primitive homology:
    \[ H^{\mathrm{prim}}_1(Y;\C)^{<g>} \iso H^{\mathrm{prim}}_1(Y/\langle g\rangle;\C) \]
  \end{lemma}
  \begin{proof}
    To see the first claim, it suffices to note that the transfer map
    $H_1(Y/\langle g\rangle;\C)\to H_1(Y;\C)$ has image in
    $H_1(Y;\C)^{<g>}$ by construction. The second claim follows from the first since
    the transfer map respects primitive homology.
  \end{proof}
  
  Now suppose that $V_i \in \Irrpr(\phi,G)$.  Let $v\in V_i$ be a vector
  so that $g_x\cdot v = v$ for some primitive $x$. If we
  identify $V_i$ as a subrepresentation of $H_1(Y;\C)$ in any way, then 
  $v \in H_1(Y;\C)^{<g_x>}$. Also note that since $H_1^{\mathrm{prim}}(Y;\C)$
  is a $G$-subrepresentation, it contains any vector $v\in V_i$ if and only if
  it contains the complete representation $V_i$. Thus we conclude the following 
  for the natural projection map $p:H_1(Y;\C)\to H_1(Y/\langle g\rangle;\C)$.  
  
  \begin{observation}[{\bf Transfer criterion}]
  \label{obs:intermediate-criterion}
  Let $V_i\in\Irr(G)$ be a subrepresentation $V_i \subset
  H_1(Y;\C)$.
  Then $V_i \subset H_1^{\mathrm{prim}}(Y;\C)$ if and only if
  $0 \neq p(v) \in H_1^{\mathrm{prim}}(Y/\langle g_x\rangle;\C)$ for
  any $v \in V_i$, and an element $g_x$ which is the image of a
  primitive element in $F_n$.
  \end{observation}
  We expect that the primitive homology of the cover $Y/\langle
  g_x\rangle \to X$ should be easier to understand than that of $Y\to
  X$, being a cover of smaller degree.  However, it is in general not
  a regular cover, and so the methods developed in this paper seem to
  be less adapted to studying it.

  To give some evidence why $H_1^{\mathrm{prim}}(Y/\langle g_x\rangle;\C)$ should be easier to
  understand, we have the following. 
  
  \begin{lemma}
  \label{lem:regular-intermediate}
    Suppose $Z \to X$ is a regular cover, and that some primitive loop
    in $X$ lifts (with degree $1$). Then
    \[ H_1^{\mathrm{prim}}(Z;\C) = H_1(Z;\C) \]
  \end{lemma}
  \begin{proof}
    Up to applying an automorphism of $F_n$, assume that $a_1$ lifts.
    Let $X'$ be the wedge of the loops $a_2,\ldots ,a_n$.  Then $Z$ is
    the union of a connected cover of $X'$ together with $|G|$ loops
    corresponding to the lifts of $a_1$.  Since $a_1w$ is primitive
    for any $w \in \pi_1(X')=F(a_2, \ldots, a_n)$, the lemma follows.
  \end{proof}
  
  In light of this we pose the following.
  
  \begin{question}
     Suppose $Z  \to X$ is a (not necessarily regular!) cover, and suppose that some primitive loop in $X$ lifts to $Z$. Is it true
     that :
    \[ H_1^{\mathrm{prim}}(Z;\C) = H_1(Z;\C)? \]  
  \end{question}
  By the discussion above, a positive answer to this question is equivalent to the 
  statement that the inclusion in Theorem~\ref{theorem:restriction} is in fact an equality.

\section{Abelian and nilpotent covers}
\label{sec:abelian}
In this section we use the point of view developed in
Section~\ref{sec:primitive} to study the primitive homology of covers
whose deck groups are abelian or $2$-step nilpotent groups.
\subsection{Abelian covers}

We begin with the following, which is an easy consequence of the
standard fact that every representation of a finite abelian group
factors through a cyclic group.
\begin{proposition}[{\bf Abelian representations}]
\label{prop:abelian}
Suppose that $n \geq 2$, and $G$ is any finite Abelian group.
Then for any homomorphism $\phi:F_n\to G$, we have $\Irrpr(\phi,G)=\Irr(G)$.
\end{proposition}
\begin{proof}
  It is enough to prove this for the case $n=2$ since we can otherwise
  simply restrict to a free factor of rank $2$. Let $V\in \Irr(G)$ be
  given, and let $\rho:G\to \mathrm{GL}(V)$ be the corresponding
  representation. Since $G$ is abelian and $V$ is irreducible, $V$ is $1$-dimensional.  Let
  $\{a,b\}$ be a free basis of $F_2$. There are numbers $n_a,n_b\geq
  0,k>0$ so that
  \[ \rho(\phi(a))(z) = e^{2\pi i n_a/k} z, \quad \rho(\phi(b))(z) =
  e^{2\pi i n_b/k} z \]
  If either of $n_a,n_b$ is zero, we are done. Otherwise,
  without loss of generality we can assume that $0<n_a \leq n_b$. Note
  that $\{a, b':= ba^{-1}\}$ is also a free basis of $F_2$, and 
  \[ \rho(\phi(ba^{-1}))(z) = e^{2\pi i (n_b-n_a)/k} z. \]
  The proposition then follows by induction on $n_b + n_a$.
\end{proof}

Proposition~\ref{prop:abelian} can be improved to the following.
\begin{proposition}\label{prop:abelian-cover}
  Let $Y\to X$ be a regular cover with finite abelian deck group $G$.  Let $X$ be a graph with 
  $\rank(\pi_1(X))\geq 2$.  Then
  \[ H^{\mathrm{prim}}_1(Y;\C) = H_1(Y;\C). \]
\end{proposition}
\begin{proof}
  Proposition~\ref{prop:abelian} gives that every $V\in
  \Irr(G)$ is also contained in $\Irrpr(G)$. Hence, we just need to
  show that the inclusion in Theorem~\ref{theorem:restriction} is an
  equality. Using Observation~\ref{obs:intermediate-criterion}, it suffices to show that
  \[ H_1^{\mathrm{prim}}(Y/\langle g\rangle;\C) =  H_1(Y/\langle g\rangle;\C) \]
  for all $g \in G$. However, since $G$ is abelian, this is simply a consequence of
  Lemma~\ref{lem:regular-intermediate}.
\end{proof}


\subsection{$2$-step nilpotent covers}

We next consider covers with finite nilpotent deck group.

\begin{proposition}[{\bf \boldmath$2$-step nilpotent representations}]
\label{prop:nilpotent}
Suppose that $n\geq 3$ and that $G$ is finite $2$-step nilpotent.
Then $\Irrpr(\phi,G)=\Irr(G)$ for any homomorphism $\phi:F_n\to G$.
\end{proposition}
\begin{proof}
Since $G$ is $2$-step nilpotent, there is an exact sequence 
  \[ 0 \to A \to N \to Q \to 0 \]
 where $A$ and $Q$ are finite abelian.  Assume that $n\geq 3$.
 By passing to a free factor if necessary, we can assume that $n=3$.   
 
 We first claim that there is a free basis $\{a,b,c\}$ of
  $F_3$ so that $\rho(\phi([a,b]))v=v$ for some $v\in V$. To see this, consider a ($1$-dimensional) irreducible subrepresentation
  $W<V$ of $A$. Note that as $G$ is $2$-step nilpotent, we have
  $\phi([a,b]), \phi([b,c]) \in A$ and thus there are
  numbers  $n_a,n_b\geq 0,k>0$ so that for all $z\in W$
  \[ \rho(\phi([a,c]))(z) = e^{2\pi i n_a/k} z, \quad \rho(\phi([b,c]))(z) =
  e^{2\pi i n_b/k} z \]
  Again using that $G$ is $2$-step nilpotent, note that
  \[ \phi([ba^{-1}, c]) = \phi([b,c])\phi([a,c])^{-1} \]
  Thus, we have that 
  \[ \rho(\phi([ba^{-1},c]))(z) = e^{2\pi i (n_b-n_a)/k} z. \]
Applying the argument as in the case when $G$ is abelian gives the claim.

Now define 
  \[ V_0 := \{ v\in V: \rho(\phi([a,b]))v=v \}. \]
  Note that $V_0\neq 0$ and $V_0$ invariant under $\rho(\phi(a)), \rho(\phi(b))$. Finally, note that the restrictions of $\rho(\phi(a))$ and $\rho(\phi(b))$ to the invariant subspace $V_0$ commute.  
  Applying the case when $G$ is abelian, we conclude that there is some primitive
  (which is a product of $a,b$) that has a nonzero fixed vector in $V_0$.
  \end{proof}

To understand if $H_1^{\mathrm{prim}}(Y;\C)=H_1(Y;\C)$ for finite nilpotent covers $Y\to X$, 
we need a somewhat more precise understanding of the representations of finite nilpotent groups. We
begin with the following, likely standard, lemma.
\begin{lemma}\label{lem:cyclic center}
  Let $G$ be a finite group and let
  \[ \rho:G \to \mathrm{GL}(V) \]
  be any irreducible representation of $G$. Then $\rho$ factors through a
  quotient of $G$ which has cyclic center and acts
  faithfully via $\rho$. 
\end{lemma}
\begin{proof}
  First note that the lemma is true for $G$ abelian: any irreducible
  representation of a finite Abelian group factors through a cyclic
  quotient. For general $G$, let $Z$ be the center of $G$. Let
  $W:=\Res^G_ZV$. Since $Z$ is central, the $Z$--isotypic components of $W$
  are $G$--subrepresentations of $V$.
  Since $V$ is irreducible as a $G$-representation,
  $W$ therefore consists of a single $Z$-isotypic component.

  By the abelian case applied to (the $Z$-irreducible summand of) $W$, there is a subgroup $K<Z$ so that $Z/K$ is cyclic and $W$ (as a $Z$-represenation) factors through a representation of $Z/K$.
  Since $Z$ is the center of $G$ we have that $K$ is normal in $G$.  Hence $V$ factors through
  a representation of $G/K$, which has cyclic center. By taking a further quotient we can assume
  that $\rho$ is faithful when restricted to the center.
\end{proof}

\begin{proposition}
\label{prop:odd-nilp-strong-irrpr}
  Let $F_{2n+1}$ be a free group of odd rank at least $3$.  Let $\phi:F_n \to G$ be 
  a surjection onto a $2$-step nilpotent group $G$ whose center has the property that each 
  of its nontrivial cyclic subgroups has prime order. Then any
  irreducible representation of $G$ factors through a quotient $H$ of
  $G$ so that the induced map $\phi:F_n\to H$ has a primitive element
  in the kernel. 
\end{proposition}
    
\begin{proof}[Proof of Proposition~\ref{prop:odd-nilp-strong-irrpr}]
  Let $V$ be an irreducible representation of $G$. By
  Lemma~\ref{lem:cyclic center} we
  can assume that $G$ has the form
  \[ 1 \to Z \to G \to Q \to 1 \]
  where $Z$ is central and cyclic. If $Z$ is trivial, we are done by
  the Abelian case. Otherwise, by our assumption on the center of $G$, the order
  of $Z$ is prime. Since $G$ is $2$-step nilpotent,
  $[G,G]\subset Z$, and therefore in fact $[G,G] = Z$.

  We first aim to show that there is a free basis $a_1, \ldots, a_{2n+1}$ of $F_{2n+1}$
  so that $\phi(a_{2n+1})$ is contained in $Z$. Namely, consider
  $\phi([a_1,a_i])$ for $i>1$. If all of these elements are trivial, then
  $\phi(a_1) \in Z$ and we are done by relabeling. Otherwise, we can
  assume that $\phi([a_1,a_2])$ is nontrivial and hence a generator of $Z$. Next, we can
  arrange that $\phi([a_1,a_i]) = 1$ for all $i>2$, by replacing $a_i$
  by $a_ia_2^{-k}$ for a suitable $k$. Similarly, we can arrange
  that $\phi([a_2,a_i]) = 1$ for all $i>2$, by replacing $a_i$
  by $a_ia_1^{-l}$ for suitable $l$. Note that 
  \[ \phi([a_1,a_ia_1^{-l}]) = \phi([a_1,a_i]) \]
  and hence after this modification we have arranged that
  \[ \phi([a_i, a_j]) = 1 \quad\quad\mbox{for any}\; 1\leq i\leq 2 < j \]
  Furthermore, note that performing Nielsen moves on $a_3, \ldots, a_{2n+1}$
  does not break this property. We can thus inductively continue, finding
  pairs $\phi([a_{2r+1},a_{2r+2}])$ etc. which are nontrivial, but so that 
  $\phi([a_i,a_k]) = 1$ for $k > 2r+2, i \leq 2r+2$. Since $2n+1$ is odd,
  after at most $n$ steps we have thus found some $a_l$ so that
  $\phi([a_i,a_l]) = 1$ for all $i$, and hence $\phi(a_l) \in K$.

  \smallskip If $\phi(a_l) = 1$ we are already done. Otherwise, since
  $G$ is $2$-step nilpotent, $\phi(a_l) \in Z = [G,G]$. Thus, there is some element $M
  \in [F_{2n+1}, F_{2n+1}]$ so that $\phi(M) = \phi(a_l)^{-1}$.

  Next, consider the $2$-step nilpotent quotient $N$ of $F_{2n+1}$. By
  naturality, we have the following commutative diagram.


  \begin{center}
    \begin{tikzcd}
      0 \arrow{r} & {[F_{2n+1}, F_{2n+1}]} \arrow{r}\arrow{d} &
      F_{2n+1}\arrow{r}\arrow{d} &
      H_1(F_{2n+1};\Z)\arrow{r}\arrow{d} & 0 \\
      0 \arrow{r} & \wedge^2 H_1(F_{2n+1};\Z)\arrow{r}\arrow{d} &
      N\arrow{r}\arrow{d} &
      H_1(F_{2n+1};\Z)\arrow{r}\arrow{d} & 0 \\
      0 \arrow{r} & Z\arrow{r} & G\arrow{r} & Q\arrow{r} & 0
    \end{tikzcd}
  \end{center}
  We denote by $[a_l]$ the image of $a_l$ in $N$, and by
  $m \in \wedge^2 H_1(F_{2n+1})$ the image of $M$ in $N$.

  Note that $[a_l]m$ is the image of a primitive element $x \in F_n$: 
  conjugating $a_l$ by $a_i$ sends $[a_l]$ to $[a_l]([a_l]\wedge[a_i])$;
  multiplying it by $[a_i,a_k]$ sends it to $[a_l]([a_i]\wedge[a_k])$.   By construction (and naturality), we then have that $\phi(x) = 1$.
\end{proof}

\begin{remark} The proof of Proposition~\ref{prop:odd-nilp-strong-irrpr}
    relies on an understanding of the image of primitive elements in
    the universal $2$-step nilpotent quotient of a free group, or
    equivalently the action of the ``Torelli group'' $\mathrm{IA}_n$
    on that quotient.  As such, it is not entirely clear if a version 
    of Proposition~\ref{prop:odd-nilp-strong-irrpr} remains true for
    groups of higher nilpotence degree (possibly increasing the rank
    of the free group). Whether or not this is the case is an
    interesting question for further research.
  \end{remark}
  
 \begin{remark} The assumption in 
Proposition~\ref{prop:odd-nilp-strong-irrpr} that the rank of the free group is odd  is crucial: it is not true 
in general that there is a primitive element which maps into the center. However, there does not seem to be a reason to suspect that the conclusion of the proposition should be false for general
    free groups.   Similarly, the condition on the center in Proposition~\ref{prop:odd-nilp-strong-irrpr} is used in the proof, but it is not clear if this assumption is really required.
 \end{remark}

We are now able to deduce the following. 

\begin{corollary}[{\bf $H_1^{\mathrm{prim}}=H^1$ for certain nilpotent covers}]
 Let $G$ be 
  a $2$-step nilpotent group $G$ whose center has the property that each 
  nontrivial cyclic subgroup has prime order. Let $Y \to X$ be a finite normal $G$-cover with 
  $\rank(\pi_1(X))=2n+1, n\geq 1$.  Then
  \[ H_1(Y;\C) = H_1^{\mathrm{prim}}(Y;\C) \]
  for any regular cover. 
\end{corollary}

\begin{proof}
  The assumption of the corollary together with
  Proposition~\ref{prop:odd-nilp-strong-irrpr} imply that some
  primitive element in $\pi_1(X)$ lifts to $Y$.  Applying
  Lemma~\ref{lem:regular-intermediate} gives the statement of the
  corollary.
\end{proof}

\section{Primitives in the kernel}
\label{sec:kernelprims}
In this section we explore criteria to detect if homomorphisms $\phi:F_n\to G$
of a free group to a finite group have primitive elements in the kernel.

As indicated in the introduction, finding $\phi$ that do not have
primitives in the kernel is a first step towards finding an example
where $\Irr(G) \neq \Irrpr(\phi,G)$. We will see that there are
various group-theoretic obstructions that prevent a group $G$ from
having such a map.

\subsection{Property $KC_i$}
Here we introduce the following notion, which is a simple (yet sometimes effective!)
tool to show that inductively constructed groups do not admit surjections without a 
primitive in the kernel.
\begin{definition}[{\bf Property ${\rm KC}_i$}]
\label{sec:KCi}
  Say that a finite group $G$ has property $\mathrm{KC}_i$
  (\emph{kernel contains corank $i$}) if for any surjection
  \[\phi: F_n \to G \]
  there is a free factor $F < F_n$ of rank at least $n-i$ contained in $\ker\phi$.
\end{definition}
We note the following easy consequence of this definition.

\begin{lemma}
\label{lem:properties-kc}
The following statements hold. 
  \begin{enumerate}
  \item Finite cyclic groups have $\mathrm{KC}_1$.
  \item Every finite group $G$ has $\mathrm{KC}_{|G|+1}$.
  \item If $K$ has $\mathrm{KC}_i$ and if  $Q$ has $\mathrm{KC}_j$, and if 
    \[ 1\to K \to G \to Q \to 1 \]
    is exact, then $G$ has $\mathrm{KC}_{i+j}$.
  \end{enumerate}
\end{lemma}
\begin{proof}
  \begin{enumerate}
  \item This is the Euclidean algorithm, as in the proof of the
    abelian case of Proposition~\ref{prop:nilpotent}.
  \item This is the pidgeonhole principle: among any $|G|+1$ elements in a free basis
    at least $2$ map to the same element in $G$; hence a Nielsen move can be applied
    to send one to the identity in $G$.
  \item This follows from the fact that free factors in free factors are free factors.
  \end{enumerate}
\end{proof}

\subsection{$p$-groups}

Recall that the  \emph{Frattini subgroup} $\Phi(G)$ of a group $G$ is defined to be the intersection of all proper maximal subgroups of $G$.  Elements of $\Phi(G)$ are often called \emph{non-generators}, since any set generating $G$ and containing such an element still generates $G$ without it. 

\begin{theorem}
\label{thm:burnside-consequences}
  Let $G$ be a $p$-group and let $F_n$ a free group with free basis $a_1,\ldots,a_n$. 
  Denote by $\pi:G \to G/\Phi(G)$ the projection map.
  Then
  \begin{enumerate}
  \item The kernel of a surjection $f:F_n \to G$ contains no primitive element
    if and only if $\{ \pi(f(a_i)), i=1,\ldots, n \}$ is a vector space basis of $G/\Phi(G)$.
  \item There is a surjection $f:F_n \to G$ whose kernel contains no primitive element
    if and only if $\dim_{\mathbb{F}_p}(G/\Phi(G)) = n$.
  \end{enumerate}
\end{theorem}

\begin{proof}
We will need the following classical result. 

\begin{lemma}[Burnside Basis Theorem]
  Let $p$ be a prime and suppose that $P$ is a $p$-group. Then $V = P/\Phi(P)$ is
  an $\mathbb{F}_p$--vector space, and :
  \begin{enumerate}
  \item A set $S \subset P$ generates $P$ if and only if its image in $V$ generates $V$.
  \item A set $S \subset P$ is a minimal generating set if and only if its image in $V$ is a
    vector space basis of $V$.
  \end{enumerate}
\end{lemma}

Recall that any primitive element $x_1 \in F_n$ can be
  extended to a free basis $x_1,\ldots, x_n$ (by definition), and any
  two free bases of $F_n$ are related by a sequence of Nielsen
  moves. As such, the images $\pi(f(x_i))$ are related to
  $\pi(f(a_i))$ by a sequence of elementary transformations. Since
  these preserve the property of being a basis, we conclude that $\{
  \pi(f(a_i)), i=1,\ldots, n \}$ is a vector space basis of
  $G/\Phi(G)$ if and only if $\{ \pi(f(x_i)), i=1,\ldots, n \}$ is a
  vector space basis of $G/\Phi(G)$ for every free basis $x_1, \ldots,
  x_n$ of $F_n$.
  
We now prove the first statement of the theorem.  If $\{ \pi(f(a_i)), i=1,\ldots, n \}$ is a vector space basis of $G/\Phi(G)$, then so 
    is $\{ \pi(f(x_i)), i=1,\ldots, n \}$ for any other free basis $x_i$. In particular, no $x_i$
    lies in the kernel of $f$. 

    Conversely, suppose that no primitive element lies in the kernel
    of $f$, but assume that $\{ \pi(f(a_i)), i=1,\ldots, n \}$ is not
    a vector space basis of $G/\Phi(G)$. Then we can assume, without
    loss of generality, that $\pi(f(a_1))$ is a linear combination of
    $\pi(f(a_i)), i > 1$. In particular, there is a word $w \in F_n$
    which only uses the letters $a_i, i>1$ so that $\pi(f(w)) = -
    \pi(f(a_1))$. Then, as $a_1w$ is primitive, there is a free basis
    $x_1,\ldots, x_n$ so that $\pi(f(x_n)) = 0$. In particular, by
    Claim 1 of the Burnside Basis Theorem, $f(x_1), \ldots, f(x_n)$ is
    not a minimal generating set of $G$ (it is a generating set since
    $f$ is surjective). However this again implies that (up to relabeling)
    $f(x_1) = f(v), v \in \langle x_2, \ldots, x_n\rangle$. Then $x_1v^{-1}$ is
    a primitive element in the kernel of $f$.
    
    The second statement of the theorem is a straightforward consequence of the first.
\end{proof}

\subsection{Connection to the Product Replacement Algorithm}
\label{sec:product-replacement}
The {\em Product Replacement Algorithm} is a common method used to
generate random elements in finite groups.  It is an active topic of
research; see \cite{Pak} or \cite{lub} for excellent surveys.  In this section we sketch a
connection between primitive elements in the kernel of a surjection
$\pi:F_n\to G$ and the product replacement algorithm for $G$.  It
seems quite likely that there are many more avenues for investigation
in this direction.

To begin, note that homomorphisms $\phi:F_n \to G$ are in
$1$--to--$1$--correspondence to $n$--tuples of elements in $G$.

A $n$--tuple $(g_1, \ldots, g_n)$ is called \emph{redundant} if there
is some $i$ so that $g_i$ is contained in the subgroup generated by 
$(g_1,\ldots,g_{-1},g_{i+1},\ldots, g_n)$.
\begin{lemma}
  If $(g_1, \ldots, g_n)$ is redundant, then the corresponding homomorphism
  $\phi:F_n \to G, \phi(a_i) = g_i$ has a primitive in the kernel. Here,
  $a_i$ is a free basis of $F_n$.
\end{lemma}
\begin{proof}
  Assume without loss of generality that
  $\phi(a_1) \in \langle \phi(a_2), \ldots, \phi(a_n)\rangle$. Let $w$ be a 
  word in $a_2, \ldots, a_n$ so that $\phi(w) = \phi(a_1)^{-1}$. Then
  $a_1w$ is a primitive element in the kernel of $\phi$.
\end{proof}
The converse to this lemma is false -- for example, consider the 
map
\[ \phi:F_2 \to \Z/6\Z, \quad \phi(a_1) = 2, \phi(a_2) = 3. \]
The set $(2,3)$ is not redundant for $\Z/6\Z$, and yet $\phi$ contains
a primitive element in the kernel ($\phi(a_1(a_2a_1^{-1})^{-2}) = 0$).

To clarify the relation between redundant elements and primitives in the kernel
of homomorphisms, we need the following definition (see e.g. \cite{Pak}).

The \emph{Product Replacement Graph} $\Gamma_n(G)$ has a vertex for
each $n$--element generating sets of $G$, and an edge connecting two
generating sets that differ by a Nielsen move (multiply one of the
elements of the tuple on the left or on the right by another element
or its inverse). The \emph{extended Product Replacement Graph}
$\widetilde{\Gamma}_n(G)$ additionally has edges corresponding to inverting an
element, or swapping two.

The question as to whether $\Gamma_n(G)$ or $\widetilde{\Gamma}_n(G)$
is connected for various values of $n$ is an extremely challenging
question; see Section~2 of \cite{Pak} for a survey of known
results. For us, the importance comes from the following, which is proved
exactly like Theorem~\ref{thm:burnside-consequences}.
\begin{lemma}\label{lemma:walking-gamma}
  Let $\phi:F_n \to G$ be a surjective homomorphism, and
  $(\phi(a_1), \ldots, \phi(a_n))$ the corresponding $n$--tuple. Then
  $\phi$ has a primitive element in the kernel if and only if the
  connected component of $\widetilde{\Gamma}_n(G)$ containing
  $(\phi(a_1), \ldots, \phi(a_n))$ also contains a redundant tuple.
\end{lemma}
\begin{corollary}
  Suppose that $\widetilde{\Gamma}_n(G)$ is connected. Then there is a
  surjection $f:F_n \to G$ with no primitive element in the kernel if
  and only if $G$ cannot be generated by fewer than $n$ elements.
\end{corollary}
In particular, connectivity results are known for solvable groups and one obtains the following Corollary of a theorem of Dunwoody \cite[Theorem~2.3.6]{Pak}
\begin{corollary}
  Suppose that $G$ is finite and solvable, and can be generated by $k$
  elements. Then, for any $n>k$ every surjection $f:F_n\to G$ has a
  primitive element in the kernel.
\end{corollary}
\begin{proof}
  By Dunwoody's theorem, $\Gamma_n(G)$ is connected as $n \geq d(G)+1$.
  Since $G$ can be generated by $k$ elements, there is a redundant generating
  $n$--tuple. Hence, we are done by Lemma~\ref{lemma:walking-gamma}.
\end{proof}

We emphasize that the conclusion of this corollary is really
nontrivial. The fact that $G$ can be generated by fewer than $n$
elements does not necessarily imply that \emph{every} $n$--element
generating set is redundant (compare again the example of $\Z/6\Z$ above). 
This discrepancy is exactly recorded by
the Product Replacement Graph.

\smallskip If the minimal size of a generating set for $G$ is $n$,
connectivity of $\Gamma_n(G)$ is particularly delicate. This seems to
be the most interesting scenario from the point of view of trying to
obtain surjections without primitives in the kernel.

\section{Free Linear Actions}
\label{sec:free-actions}

In this section we discuss groups which act freely and linearly on spheres.
The importance of such groups to our goal stems from the following.
\begin{lemma}
  Suppose that $G$ acts freely and linearly on some sphere. Then for the
  cover $Y\to X$ of a $n$-petal rose $X$ defined by $\phi:F_n \to G$ we
  have
  \[ H_1^\mathrm{prim}(Y;\C) = H_1(Y;\C) \quad\Leftrightarrow\quad \ker(\phi)\,\mbox{contains a primitive element} \]
\end{lemma}
\begin{proof}
  Suppose that $\ker(\phi)$ does not contain a primitive
  element. Then, by assumption on $G$, there is a irreducible
  representation $V$ of $G$ where no element $1 \neq g$ fixes any
  vector. Hence, $V \neq \Irrpr(\phi, G)$, and therefore
  $H_1^\mathrm{prim}(Y;\C) \neq H_1(Y;\C)$ by
  Theorem~\ref{theorem:restriction}. The other direction is immediate
  from Lemma~\ref{lem:regular-intermediate}.
\end{proof}

Using the methods developed above we can show that (at least for large $n$)
this does indeed always occur. Namely, we have the following.

\begin{proposition}\label{prop:prim-in-ker}
  There is a number $B\geq3$ with the following property. Suppose that $G$ is 
  any group which admits a (complex) representation in which no non-identity
  element has a fixed vector. If $n\geq B$ and $\phi:F_n \to G$ is any surjection, then the
  kernel of $\phi$ contains a primitive element in $F_n$. In other words, $\Irr=\Irrpr$ for
  all such groups, and $H_1^{\mathrm{prim}}(Y;\C) = H_1(Y;\C)$ for the associated covers.
\end{proposition}
The proof of this ultimately relies on the classification of finite groups which act 
freely and linearly on spheres. 
\begin{proof}[Proof of Proposition~\ref{prop:prim-in-ker}]
  Let $G$ be as in the proposition.  By the classification of finite
  groups acting freely linearly on spheres (compare page 233 of
  \cite{DM} for the definition of types, and Theorem 1.29 to restrict
  the quotients in cases V, VI), there is a normal subgroup $N<G$ with
  the following properties:
  \begin{enumerate}
  \item $N$ is metabelian, i.e. an split extension of a cyclic by a cyclic group.
  \item $G/N$ is either an extension of a cyclic by a cyclic group, or equal to one
    out of a list of four possible finite groups $Q_1, \ldots, Q_4$.
  \end{enumerate}
  Thus the existence of the number $B$ as desired follows from Lemma~\ref{lem:properties-kc}.
\end{proof}

\begin{remark}
  While the proof of Proposition~\ref{prop:prim-in-ker} given above is nonconstructive,
  we expect the minimal value for $B$ to be fairly low (in fact, $3$ or $4$). Computer experiments
  (see below) are in agreement with this.
\end{remark}

\section{Algorithms}
\label{sec:algorithms}
In this section we explain that the question if $\Irr(G) = \Irrpr(\phi,G)$ for
any given finite group $G$ and surjection $\phi:F_n\to G$ is algorithmic, supposing that the
irreducible representations are understood. The main ingredient here
is the following:

\begin{proposition}
\label{prop:alg}
  Given a homomorphism $q:F_n \to G$ of a free group to a finite group $G$, there
  is an algorithm which computes the set of all elements in $G$ which are the image of 
  primitive elements in $F_n$ in finite time.
\end{proposition}
\begin{proof}
  To begin, note that every primitive element in $F_n$ is the image of the first standard
  free generator of $F_n$ under some element of $\Aut(F_n)$. Additionally, $\Aut(F_n)$ 
  is generated by Nielsen moves. Thus suggests the following algorithm:
  \begin{itemize}
  \item Start with a set $L$ of elements in $G^n$, containing the images of the standard basis
    $q(a_1),\ldots, q(a_n)$ under $q$.
  \item For every element in $L$, apply all basic Nielsen moves to the tuple. Add all
    resulting tuples to $L$ (unless they were already contained in $L$).
  \item If $L$ grew in size during the last step, repeat the last step.
  \item The desired set of images is now the set of all entries of tuples in $L$.
  \end{itemize}
  To see that the resulting list actually contains every image of a
  primitive, note that if $g=q(x)$ for some primitive, then there is
  some sequence $M_1\to M_2\to \dots \to M_k$ of Nielsen moves of the
  standard free basis $a_1,\ldots, a_k$ ending in a basis
  $x_1,\ldots,x_k$ containing $x$. Now observe that by construction
  every image of $q(a_1),\ldots, q(a_k)$ under any sequence of moves
  $M_i$ is contained in $L$; thus $q(x)$ is one of the entries of a
  tuple in $L$.
\end{proof}
We stress that we do not claim that the algorithm is particularly fast. Also note that
\cite{CG} contains a different algorithm that works for non-regular covers, but seems
harder to implement in practice.

\smallskip
For groups as discussed in Section~\ref{sec:free-actions} experiments with the 
above algorithm show that already for a free group of rank $3$, any surjection 
contains a primitive element for all groups of order at most $~1000$ that can act
freely on spheres.

\smallskip
With Proposition~\ref{prop:alg} in hand, the claim made at the beginning of this
section is immediate: given $\phi:F_n\to G$ one can first compute the images of
all primitive elements of $F_n$ under $\phi$ in $G$. For each irreducible representation
one can then check if any of the associated matrices have eigenvalue $1$.

\section{Rank $2$ examples}
\label{sec:rank-2-examples}
In this section we give various examples with $\rank(\pi_1(X))=2$ and where $\Irrpr(\phi,G)\subsetneq\Irr(G)$ and $H_1^{\mathrm{prim}}(Y;\C)\subsetneq H^1(Y;\C)$.  These show in particular that our results assuming $\rank(\pi_1(X))\geq 3$ are sharp.

\subsection{A group acting freely on a sphere}

\begin{proposition}\label{prop:examples-primitive}
  There are surjections $\phi:F_2 \to G$ with the following two properties:
  \begin{enumerate}[i)]
  \item The group $G$ acts freely and linearly on a sphere $S^N$.
  \item No primitive element of $F_2$ is contained in the kernel of $\phi$.
  \end{enumerate}
  In particular, for the associated cover we have $H_1^{\mathrm{prim}}(Y;\C)\neq H^1(Y;\C)$. 
\end{proposition}

\begin{proof}
Consider the group $G$ defined by the relations
\[ G = \langle a, b | a^3=1, b^8=1, bab^{-1}=a^2 \rangle \] 
This group appears as Type i), with parameters $n = 8, m = 3, r = 2$
in the list in \cite{N}, and hence acts linearly and freely on some sphere.

Using the algorithm described in Section~\ref{sec:algorithms}, one can check that
the map $q:F(a,b)\to G$ has no primitive element in the kernel. Hence $H_1^{\mathrm{prim}}(Y;\C)\neq H^1(Y;\C)$ for the associated cover. 
\end{proof}

\subsection{A $2$-step nilpotent group}

We now give examples which show that the obstruction given by $\Irrpr(\phi,G)$
is indeed stronger than requiring that $G$ act linearly and freely on spheres.
\begin{proposition}\label{prop:examples-primitive2}
  There is a surjection $\phi:F_2 \to G$, and a representation $\rho$ of
  $G$ with the following properties:
  \begin{enumerate}[i)]
  \item No primitive element of $F_2$ is contained in the kernel of $\phi$.
  \item For every primitive element $x\in F_2$, the matrix $\rho(\phi(x))$ does not have
    eigenvalue $1$.
  \item There is some some $y\in F_n$, so that $\rho(\phi(y))$ is the identity.
  \end{enumerate}
  In particular, we have $\Irrpr(\phi, G) \neq \Irr(G)$ and 
  for the associated cover we have $H_1^{\mathrm{prim}}(Y;\C)\neq H^1(Y;\C)$,
  yet $G$ does not act freely and linearly on a sphere via $\rho$.
\end{proposition}

\begin{proof}
Consider the $2$-step nilpotent group $\Gamma$ which fits into the sequence
\[ 1 \to \Z/2\Z \to \Gamma \to \Z/4\Z \times \Z/4\Z \to 1\]
and is obtained from the canonical mod-$4$, $2$-step nilpotent quotient of $F_2$ by taking 
a further rank $2$ quotient of the center.

Thus, there is a surjection $\phi:F_2 \to \Gamma$ which induces the mod-$4$-homology
map $F_2\to H_1(F_2;\Z/4\Z)$ in abelianizations. $\Gamma$ is generated by
two elements $\alpha,\beta$ which are images of a free basis $a,b$ of $F_2$.

Consider the representation
\[ \rho: \Gamma \to GL(\C^2) \]
given by 
\[ \alpha \mapsto
\begin{pmatrix}
  i & 0 \\ 0 & -i
\end{pmatrix}, \quad \beta\mapsto
\begin{pmatrix}
  0 & 1 \\ -1 & 0 
\end{pmatrix} \]
Applying the algorithm described in Section~\ref{sec:algorithms}, one checks that
if $x\in F_2$ is primitive, then $\rho(\phi(x))$ does not have eigenvalue $1$. Hence,
the representation $\rho$ is not an element of $\Irrpr(\phi,\Gamma)$. However,
$\phi(a^2[a,b]) \neq 1$, yet $\rho(\phi(a^2[a,b])) = 1$, so $\rho$ is not faithful. In particular,
$\Gamma$ does not act on a sphere linearly and freely via $\rho$.
This proves the proposition.
\end{proof}

\subsection{Torus homology cover}
The goal of this section is to give a more geometric construction of a
torus cover with nontrivial primitive homology. Namely, we will show
\begin{proposition}
  There is a cover $Y\to T$ of the torus $T$ with one boundary component,
  which is obtained as an iterated homology cover, and
  where
  \[ H_1^\mathrm{prim}(Y;\C) \neq H_1(Y;\C) \]
\end{proposition}

Choose a basepoint on the torus $T^2$ and identify the fundamental
group based at that point with the free group on $A,B$. We first consider 
the mod-$2$ homology cover $X\to T^2$. Denote by $\alpha, \beta$ the deck group
elements corresponding to $A,B$ respectively.
\bigskip

\includegraphics[width=0.5\textwidth]{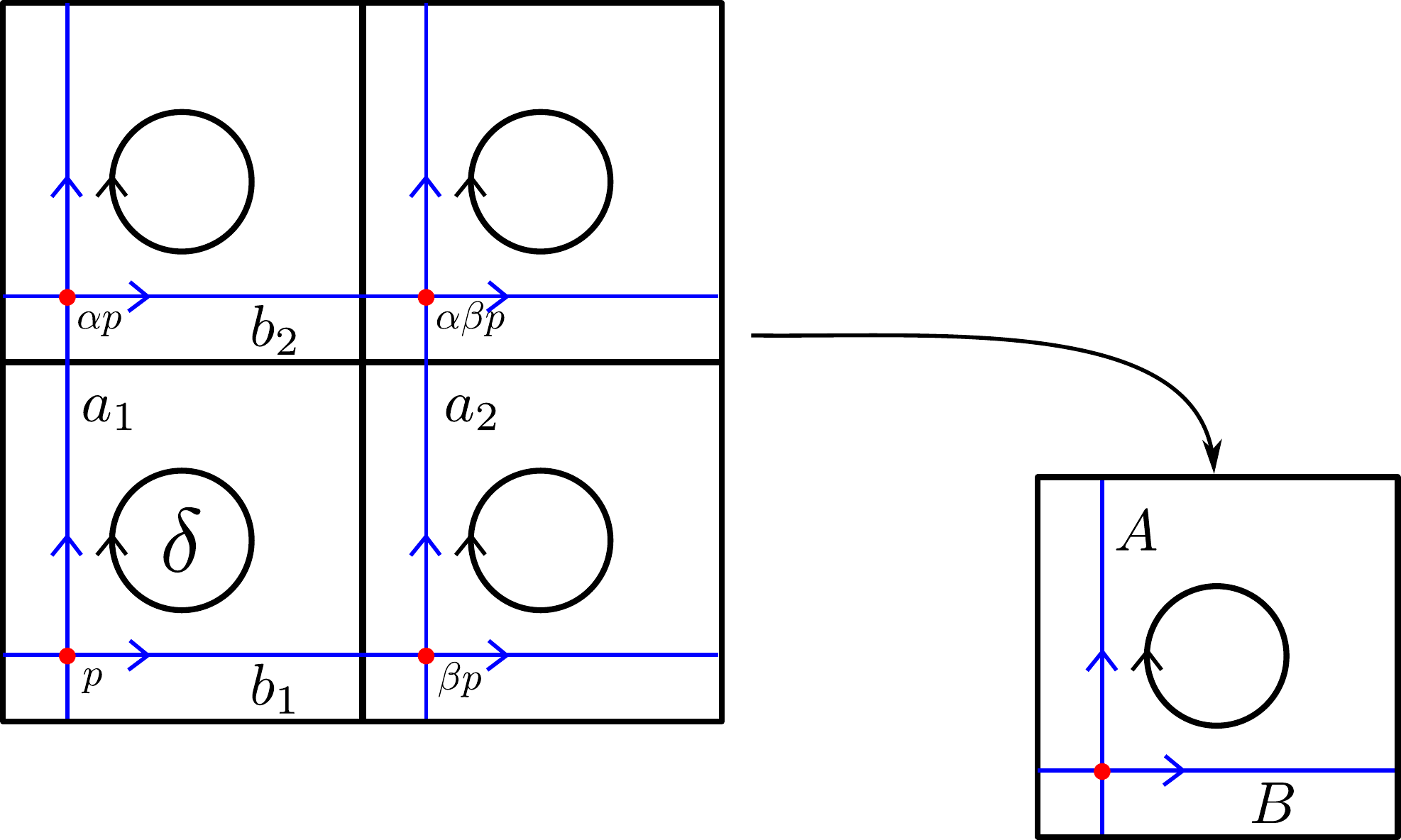}

The homology $H_1(X;\Z)$ is isomorphic to $\Z^5$; we choose an
explicit basis
$$a_1, a_2, b_1, b_2, \delta$$
where $a_1, a_2$ are the two components of the preimage of $A$, and
$b_1,b_2$ are the two components of the preimage of $B$, and $\delta$
is a lift of the boundary component (to the preferred basepoint).

\begin{lemma}
  This is indeed a basis.
\end{lemma}
\begin{proof}
We show that is a generating set. Namely, we have
$a_1 - a_2 = \delta + \alpha\delta$, and hence $\alpha\delta$ is in
the span $V$ of the set in question. Similarly, $\beta\delta$ is in
$V$, and since $\delta+\alpha\delta+\beta\delta+\alpha\beta\delta = 0$
all four boundary components are in $V$. Furthermore, $a_1,b_1$
correspond two curves intersecting once, and hence $V$ is everything.
\end{proof}

\begin{lemma}
  The deck transformations $\alpha, \beta$ act as vertical
  (resp. horizontal) translations on $X$. Their matrices with respect
  to our basis are
  $$\alpha =
  \begin{pmatrix}
    1 & 0 & 0 & 0 & 1 \\
    0 & 1 & 0 & 0 & -1 \\
    0 & 0 & 0 & 1 & 0 \\
    0 & 0 & 1 & 0 & 0 \\
    0 & 0 & 0 & 0 & -1 
  \end{pmatrix}, 
  \quad\beta = 
  \begin{pmatrix}
    0 & 1 & 0 & 0 & 0 \\
    1 & 0 & 0 & 0 & 0 \\
    0 & 0 & 1 & 0 & -1 \\
    0 & 0 & 0 & 1 & 1 \\
    0 & 0 & 0 & 0 & -1
  \end{pmatrix}
  $$
\end{lemma}
\begin{proof}
  As $\alpha$ acts as a vertical translation, it fixes both $a_1,a_2$
  and interchanges $b_1,b_2$. We have
  $$\delta + \alpha\delta + a_2 - a_1 = 0$$
  as those four curves are the boundary of the ``left'' two-holed annulus.
  This explains the first matrix. Similarly, $\beta$ acts as a
  horizontal translation. Thus, $b_1, b_2$ are fixed and $a_1, a_2$
  are exchanged. We have
  $$\delta + \beta\delta + b_1 - b_2 = 0$$
  as those four are the boundary of the ``lower'' two-holed
  annulus. This explains the second matrix.
\end{proof}

To analyse which elements of homology are components of lifts of
simple closed curves, we use the following lemma.
\begin{lemma}\label{lem:identify-ind}
  If $x \in H_1(X;\Z/p^n\Z)$ is equal to a multiple of a component of
  the preimage of a nonseparating simple closed curve, then
  one of $\alpha x, \beta x, \alpha\beta x$ is equal to $x$ and
  $$\dim_{\mathbb{F}_{p^n}}\mathrm{Span}\{x, \alpha x, \beta x,
  \alpha\beta x\} = 2.$$
\end{lemma}
\begin{proof}
  Suppose that $x = [\widetilde{\gamma}]$ is defined by a component
  $\widetilde{\gamma}$ of the preimage of a simple closed curve
  $\gamma$. First note that there is a mapping class $\varphi$ of
  $T^2$ and a lift $\widetilde{\varphi}$ so that $\widetilde{\varphi}
  = a_1$. Namely, we can choose $\varphi$ so that $\varphi(\gamma) =
  A$. Then, for any lift $\widetilde{\varphi}$ we have that
  $\widetilde{\varphi}(\widetilde{\gamma})$ is a component of the
  preimage of $A$. We can choose a lift so that this component is
  $a_1$.

  Since $\widetilde{\varphi}$ is equivariant with respect to the deck
  group action, it therefore maps the $G$-span of $x$ to the $G$-span
  of $a_1$. In particular, the $G$-span of $x$ has rank 2. Since $a_1$
  is fixed by $\alpha$, we have that $x$ is fixed by
  $\widetilde{\varphi}_*^{-1}(\alpha)$ as claimed.
%
\end{proof}

Consider an element
$$x = r_1 a_1 + r_2 a_2 + s_1 b_1 + s_2 b_2 + d \delta$$
and its orbit under $G$: 
\begin{eqnarray*}
  \alpha x &=& (r_1+d) a_1 + (r_2-d) a_2 + s_2 b_1 + s_1 b_2 -d \delta\\
  \beta x &=& r_2 a_1 + r_1 a_2 + (s_1-d) b_1 + (s_2+d) b_2 - d\delta \\
  \alpha\beta x &=& (r_2-d) a_1 + (r_1+d) a_2 + (s_2+d) b_1 + (s_1-d)
  b_2 + d\delta
\end{eqnarray*}
We collect the coefficients in the matrix:
$$
\begin{pmatrix}
  r_1 & r_2 & s_1 & s_2 & d \\
  r_1 + d & r_2 - d & s_2 & s_1 & -d \\
  r_2 & r_1 & s_1-d & s_2+d & - d \\
  r_2 - d & r_1 + d & s_2 +d & s_1-d & d
\end{pmatrix}
$$
If $x$ is defined by a component of the lift of a simple closed curve,
then by Lemma~\ref{lem:identify-ind} the first row has to be equal to one
of the other rows. We distinguish cases, depending on which rows are equal.

\paragraph{Row 1 equal to Row 2}
This immediately implies $d=0$ and then $s_1 = s_2$. This leaves us
with the matrix
$$
\begin{pmatrix}
  r_1 & r_2 & s_1 & s_1 & 0 \\
  r_1 & r_2 & s_1 & s_1 & 0 \\
  r_2 & r_1 & s_1 & s_1 & 0 \\
  r_2 & r_1 & s_1 & s_1 & 0
\end{pmatrix}
$$
which needs to have rank $2$. There are two subcases: if $s_1 \neq 0$
then the matrix has rank $2$ if and only if $r_1 \neq r_2$. If $s_1 =
0$, then the matrix has rank $2$ if and only if $r_1 \neq \pm r_2$.
\begin{framed}
  $d=0$, $s_1=s_2$, $r_1 \neq r_2$. If $s_1=s_2=0$ then also $r_1 \neq -r_2$.
\end{framed}

\paragraph{Row 1 equal to Row 3}
This immediately implies $d=0$ and then $r_1 = r_2$. This leaves us
with the matrix
$$
\begin{pmatrix}
  r_1 & r_1 & s_1 & s_2 & 0 \\
  r_1 & r_1 & s_2 & s_1 & 0 \\
  r_1 & r_1 & s_1 & s_2 & 0 \\
  r_1 & r_1 & s_2 & s_1 & 0
\end{pmatrix}
$$
which needs to have rank $2$. There are two subcases: if $r_1 \neq 0$
then the matrix has rank $2$ if and only if $s_1 \neq s_2$. If $r_1 =
0$, then the matrix has rank $2$ if and only if $s_1 \neq \pm s_2$.
\begin{framed}
  $d=0$, $r_1=r_2$, $s_1 \neq s_2$. If $r_1=r_2=0$ then also $s_1 \neq -s_2$.
\end{framed}

\paragraph{Row 1 equal to Row 4}
This leads to $r_1 = r_2 - d$ and $s_1 = s_2 + d$. Simplifying the
matrix yields
$$
\begin{pmatrix}
  r_1 & r_1 + d & s_1 & s_1 - d & d \\
  r_1 + d & r_1 & s_1-d & s_1 & -d \\
  r_1 + d & r_1 & s_1-d & s_1 & - d \\
  r_1 & r_1 + d & s_1 & s_1 - d & d 
\end{pmatrix}
$$
If $d=0$ this has rank $1$, so this case never happens. The matrix has
rank $\leq 1$ if and only if the first two rows are dependent. As $d
\neq 0$ the only way this can happen is if the second row is the
negative of the first row. This leads to
$$-r_1 = r_1 + d, -s_1 = s_1-d$$
Hence the matrix has rank $2$ if one of the above is false, leading to

\begin{framed}
  $d\neq 0$, $r_1=r_2-d$, $s_1 = s_2+d$ and $d \neq -2r_1$ or $d \neq 2s_1$.
\end{framed}

\paragraph{Consequence}
\begin{lemma}
  The representation of $(\Z/p)^5$ defined by
  $$\rho(r_1, r_2, s_1, s_2, d) = \left( z \mapsto \zeta_p^{r_1-r_2 +
      s_1-s_2 + d}z \right)$$
  has the property that no $x\in (\Z/p)^5$ which is equal to the
  component of a preimage of a simple closed curve acts as the identity.
\end{lemma}
\begin{proof}
 We consider in turn the three possibilities for such $x$. All
 exponents are seen as elements of $\Z/p$. The equalities and
 inequalities also are in this field.
 \begin{enumerate}
 \item $d=0$, $s_1=s_2$, $r_1 \neq r_2$. If $s_1=s_2=0$ then also $r_1
   \neq -r_2$.

   In that case the element acts as 
   $$z \mapsto \zeta_p^{r_1-r_2 + s_1-s_2 + d} = \zeta_p^{r_1-r_2} z
   \neq z$$
 \item $d=0$, $r_1=r_2$, $s_1 \neq s_2$. If $r_1=r_2=0$ then also $s_1
   \neq -s_2$. 

   In that case the element acts as 
   $$z \mapsto \zeta_p^{r_1-r_2 + s_1-s_2 + d} = \zeta_p^{s_1-s_2} z
   \neq z$$
 \item $d\neq 0$, $r_1=r_2-d$, $s_1 = s_2+d$ and $d \neq -2r_1$ or $d
   \neq 2s_1$.

   In that case the element acts as
   $$z \mapsto \zeta_p^{r_1-r_2 + s_1-s_2 + d} = \zeta_p^{d} z
   \neq z$$
 \end{enumerate}
\end{proof}
The following corollary finishes the proof of the main proposition of
this section.
\begin{corollary}
  Consider the group $G$ defined by the extension
  \[ 1 \to (\Z/p)^5 \to G \to (\Z/2)^2 \to 1 \]
  via iterated homology covers of $G$, and the corresponding $\phi:F_2\to G$.
  Then there is a irreducible representation $V$ of $G$ so that no
  image of a simple closed curve in $G$ fixes any vector.
\end{corollary}
\begin{proof}
  We consider the representation induced from the representation
  $\rho$ of $(\Z/p)^5$ to $G$. While this may not be irreducible, it does have
  the property that no image of a simple closed curve has any fixed vectors.
  Namely, suppose that $v$ would be fixed by the image of $\gamma$. Then, 
  for $\gamma^n$ the smallest power so that $\gamma^n$ maps into $(\Z/p)^5$,
  a component of $v$ would also have a fixed vector. By the previous lemma,
  this is impossible unless $v = 0$.
\end{proof}

\section{The case of surfaces and simple closed curves}
\label{section:surfaces}
\label{sec:surfaces}

There is a close analogy of our discussion for finite regular covers
$f:Y\to X$ where $X$ is a genus $g\geq 2$ surface.  In this section we sketch this analogy.

\subsection{Simple homology}
The analog for surfaces of primitive homology for graphs is the following. 

\begin{definition}[{\bf simple homology}]
Let $f:Y\to X$ be a finite cover.  The \emph{simple closed curve homology} (or \emph{simple homology})  $H_1^{\mathrm{scc}}(Y;\C)$ corresponding to this finite cover is defined to be:
\[H_1^{\mathrm{scc}}(Y;\C):=\text{$\C$-span}\{W_\gamma: \text{$\gamma\in \pi_1(X)$ is a simple closed curve}\}\]
\end{definition}
Using work of Putman-Wieland \cite{PW}, Boggi-Looijenga \cite{BL} conjecture
that vanishing of the virtual first Betti numbers of the moduli spaces of Riemann surfaces 
(which is a well-known open question of Ivanov; see \cite{PW} for precise
statements) is closely related to the question if  $H_1(Y;\C)=H_1^{\mathrm{scc}}(Y;\C)$. 

Again, the very basic question if simple homology is ever a proper 
subrepresentation is at the current point wide open (in the case of integral
coefficients, examples have been constructed where simple homology 
has finite index bigger than one, see \cite{Ir, KR}).

\smallskip
Here, we want to again study a representation-theoretic version of simple homology.

\begin{definition}[{\boldmath$\Irrscc(\phi,G)$}]
  Let $G$ be a finite group.  Fix a homomorphism $\phi: \pi_1(X)\to
  G$.  Let $\Irrscc(\phi,G)$ denote the set of irreducible representations
  $V$ of $G$ with the property that $\phi(\gamma)(v)=v$ for some
  element $\gamma$ which can be represented by a simple closed curve
  and some nonzero $v\in V$.
\end{definition}

We will prove the analog of Theorem~\ref{theorem:restriction}
for surfaces.

\begin{theorem}[{\bf Restricting simple closed curve homology}]
\label{thm:characterize-scc}
Let $f:Y\to X$ be a regular $G$-covering of surfaces defined by
$\phi:\pi_1(X)\to G$, where $G$ is a finite group and $X$ has genus at
least $2$.  Then
\[ H_1^{\mathrm{scc}}(Y;\C) \subseteq \bigoplus_{V_i\in \Irrscc(\phi,G)}
V_i^{(2g-2)\cdot\dim(V_i)} \oplus \C^2_{triv}\]
\end{theorem}
For abelian covers it is again not hard so show that simple homology is all of homology
(this follows e.g. from the considerations in \cite{Lo}) .

As in the case of primitive homology one can perform computer experiments to test if
covers satisfy the representation--theoretic obstruction for groups acting
freely and linearly on spheres. The examples from Propostion~\ref{prop:examples-primitive}
still apply; but the case of simple closed curves is more restrictive than the
primitive case, as the following shows (compare the end of the section for a sketch of the proof):
\begin{proposition}\label{prop:torus-scc-prim}
  Let $\Sigma_{1,2}$ be a torus with two marked points.
  There are surjections $\phi:\pi_1(\Sigma_{1,2}) \to G$ with the following two properties:
  \begin{enumerate}[i)]
  \item The group $G$ acts freely and linearly on a sphere $S^N$.
  \item No element of $\pi_1(\Sigma_{1,2})$ representing a simple closed curve
    is contained in the kernel of $q$.
  \item There are primitive elements contained in the kernel of $\phi$.
  \end{enumerate}
  In particular, we have $\Irrscc(\phi,G)\neq\Irr(G)$ and
  for the associated cover we have $H_1^{\mathrm{scc}}(Y;\C)\neq H_1^{\mathrm{prim}}(Y;\C)=H^1(Y;\C)$. 
\end{proposition}
On the contrary, already for closed genus $2$ surfaces we were unable to find
with computer experiments an example with $\Irrscc(\phi,G)\neq\Irr(G)$. Note that
in contrast to the primitive case it is not clear if this condition suffices to ensure
that $H_1^{\mathrm{scc}}(Y;\C)=H^1(Y;\C)$. 

\begin{proof}[Proof of Theorem~\ref{thm:characterize-scc}]
We follow with the same argument as in
Section~\ref{sec:primitive}, with the following lemma
replacing Step 1b.
\begin{lemma}
  Let $F$ be a closed surface with no boundary components.
  Let $\hat{\alpha}_1, \ldots, \hat{\alpha}_n$ be disjoint pairwise
  nonisotopic simple closed curves on $F$. Let $R_1, \ldots, R_m$ be
  the set of subsurfaces bounded by $\hat{\alpha}_1 \cup \ldots \cup
  \hat{\alpha}_n$. We have an exact sequence
  \[ 
  0 \to \C\left[\sum \partial R_j\right] \to \bigoplus_{j=1}^m
  \C\left[\partial R_j\right] \to
  \bigoplus_{i=1}^n\C\left[\hat{\alpha}_i\right] \to
  \mathrm{Span}_{H_1(X)}\left\{ \hat{\alpha}_1, \ldots, \hat{\alpha}_n \right\}
  \to 0
  \]
  where the maps are induced by the natural identifications of
  multicurves in the various direct sums.
\end{lemma}
\begin{proof}
  We check exactness at the various places individually. The
  surjectivity of $\phi:\bigoplus_{i=1}^n\C\left[\hat{\alpha}_i\right] \to
  \mathrm{Span}_{H_1(X;\C)}\left\{ \hat{\alpha}_1, \ldots, \hat{\alpha}_n
  \right\}$ is clear by definition.

  Next, consider the image of $\psi:\bigoplus_{j=1}^m
  \C\left[\partial R_j\right] \to
  \bigoplus_{i=1}^n\C\left[\hat{\alpha}_i\right]$. The fact that the
  image of $\psi$ lies in the kernel of $\phi$ follows since each
  generator $\partial R_j$ of $\C\left[\partial R_j\right]$ is mapped
  to a boundary by definition. 

  To see that the image of $\psi$ is the complete kernel of $\phi$,
  suppose that some linear combination of the $\hat{\alpha}_i$ is zero
  in homology:
  $$a_1[\hat{\alpha}_1] + \ldots + a_n[\hat{\alpha}_n] = 0$$
  First note that if the multicurve $\hat{\alpha}_1 \cup \ldots \cup
  \hat{\alpha}_n$ is nonseparating, then all $a_i$ are zero, and there
  is nothing to show. 

  If not, then our goal is to use the relations $\psi(\partial R_j)$
  to modify the relation so that the involved curves become
  nonseparating.  To this end, consider the subsurface $R_1$. We can
  assume that there is a curve $\hat{\alpha}_i$ which appears only
  once in the boundary of $R_1$; if there is no such curve, then $R_1$
  itself is the complement of $\hat{\alpha}_1 \cup \ldots \cup
  \hat{\alpha}_n$ and the latter is nonseparating.

  Let now $R_2$ be the subsurface on the other side of
  $\hat{\alpha}_i$.  Using the element $\psi(\partial R_2)$ to modify
  the linear combination, we can change the coefficient $a_i$
  of $\hat{\alpha}_i$ to be $0$. We interpret the resulting linear
  combination as one involving only the curves $\hat{\alpha}_i, j \neq
  i$, where now $R_1 \cup R_2$ in place of $R_1$ is one of the complementary
  subsurfaces of the relation
  $a_2[\hat{\alpha}_2] + \ldots + a_n[\hat{\alpha}_n] = 0$. Induction
  on the complexity of $R_1$ now yields the claim.

  \smallskip
  Injectivity of $\rho:\C\left[\sum \partial R_j\right] \to \bigoplus_{j=1}^m
  \C\left[\partial R_j\right]$ is clear. 
  The fact that the image of $\rho$ is exactly the kernel of $\psi$
  follows since each $\hat{\alpha}_i$ appears exactly twice amongst
  the curves appearing in $\partial R_j$ (with opposite signs): 
  on the one hand this immediately implies that the sum of the $\partial R_j$ maps to
  $0$ under $\psi$. On the other hand, consider any linear combination
  of a proper subset of the $\partial R_j$. In that case, there is at
  least one $\hat{\alpha}_i$ which appears in exactly one of the
  $\partial R_j$ of the subset. For the combination to be in the kernel of $\psi$
  the coefficient of that $\partial R_j$ needs to be $0$. Inductively
  it now follows that the combination is trivial.
\end{proof}

This completes the proof of the theorem.
\end{proof}

\subsection{Algorithms}
It is a basic fact in surface topology that every nonseparating simple closed curve is 
the image of the first standard generator under a suitable element of the mapping class
group. Since the mapping class group is generated by finitely many (explicit) elements,
one can adapt the algorithm given in Section~\ref{sec:algorithms} to yield the following.

\begin{proposition}
  Given a homomorphism $q:\pi_1(\Sigma) \to G$ of the fundamental group of a surface to
  a finite group $G$, there is an algorithm which computes the set of all elements in $G$ which 
  are the image of elements representable by nonseparating simple closed curves.
\end{proposition}

Let $\Sigma_{1,2}$ be a torus with two marked points. Define a group $G$ via the relations
\[ G = \langle a, b, c | a^3=1, b^4=1, bab^{-1}=a^2, c^2=b^2,
cac^{-1}=a, cbc^{-1}=b^3 \rangle \] 
This group appears as Type ii), with parameters $n = 4, m = 3, r = 2, l = 1, k = 3$
in the list in \cite{N}, and hence acts linearly and freely on some sphere.

We choose a basis $a,b,c$ of
$\pi_1{\Sigma_{1,2}}$ of simple closed curves, where $a$ surrounds one
of the punctures, and $b,c$ generate $H_1(\Sigma_{1,2})$. Let
$q:\pi_1(\Sigma_{1,2}) \to G$ be the obvious map induced by the labels.

Using the algorithms described, one can verify:
\begin{enumerate}[i)]
  \item The group $G$ acts freely and linearly on a sphere $S^N$.
  \item No element of $\pi_1(\Sigma_{1,2})$ representing a simple closed curve
    is contained in the kernel of $q$.
  \item There are primitive elements contained in the kernel of $q$.
\end{enumerate}
  
In particular, for the associated cover we have 
\[H_1^{\mathrm{scc}}(Y;\C)\neq H_1^{\mathrm{prim}}(Y;\C)=H^1(Y;\C)\]
This shows Proposition~\ref{prop:torus-scc-prim}.

\small

\begin{tabular}{ll}
Benson Farb & Sebastian Hensel\\
Department of Mathematics & Mathematisches Institut\\ 
University of Chicago & Rheinische Friedrich-Wilhelms-Universit\"at Bonn\\
5734 University Ave. & Endenicher Allee 60\\
Chicago, IL 60637 & 53115 Bonn, Germany\\
E-mail: farb@math.uchicago.edu & E-mail: hensel@math.uni-bonn.de
\end{tabular}
\end{document}